\documentclass[11pt,twoside]{article}

\newcommand{\HeadTitle}{Cotangent Coefficients for Mittag--Leffler Polynomials}

\newcommand{\HeadTitleTwo}{\begin{center}
\Large{\textit{From Umbral Hyperbolic Integrals to a Cotangent Coefficient Formula for the Mittag--Leffler Polynomials}}
\end{center}}

\usepackage[margin=0.8in]{geometry}
\usepackage{amsmath, amssymb, amsfonts, mathtools}
\usepackage{amsthm}
\usepackage{mathrsfs}
\usepackage{stmaryrd}
\usepackage{esint}
\usepackage{eqnalign}

\usepackage{graphicx}
\usepackage{enumitem}
\usepackage[most]{tcolorbox}
\usepackage{float}
\usepackage{listings}
\usepackage{array}
\usepackage{booktabs}
\usepackage{xcolor}
\usepackage{emptypage}
\usepackage{tensor}
\usepackage{tabularx}

\usepackage{fancyhdr}
\usepackage{hyperref}
\usepackage[nameinlink,noabbrev]{cleveref}

\usepackage{titlesec}
\usepackage[T1]{fontenc}
\usepackage[utf8]{inputenc}
\usepackage{lmodern}
\usepackage{titling}

\usepackage[backend=biber,style=numeric]{biblatex}
\providecommand{\HeadTitle}{} 
\providecommand{\HeadTitleTwo}{} 
\providecommand{\HeadAuthor}{Luc Ramsès TALLA WAFFO} 

\titleformat{\section}[block]
  {\normalfont\large\bfseries\itshape\centering}
  {§\thesection.}
  {1em}
  {}

\titleformat{\subsection}
  {\normalfont\normalsize\bfseries}
  {\thesubsection}
  {1em}
  {}

\newtheorem{theorem}{Theorem}[section]
\newtheorem{lemma}[theorem]{Lemma}
\newtheorem{proposition}[theorem]{Proposition}
\newtheorem{corollary}[theorem]{Corollary}

\newtheorem{remark}[theorem]{Remark}

\crefname{example}{example}{examples}
\Crefname{example}{Example}{Examples}

\crefname{corollary}{corollary}{corollaries}
\Crefname{corollary}{Corollary}{Corollaries}

\crefname{definition}{definition}{definitions}
\Crefname{definition}{Definition}{Definitions}

\crefname{remark}{remark}{remarks}
\Crefname{remark}{Remark}{Remarks}

\crefname{conjecture}{conjecture}{conjectures}
\Crefname{conjecture}{Conjecture}{Conjectures}

\crefname{lemma}{lemma}{lemmas}
\Crefname{lemma}{Lemma}{Lemmas}

\crefname{proposition}{proposition}{propositions}
\Crefname{proposition}{Proposition}{Propositions}

\crefname{theorem}{theorem}{theorems}
\Crefname{theorem}{Theorem}{Theorems}

\numberwithin{equation}{section}

\usepackage{fontspec}
\begin{document}

\thispagestyle{fancy}

\vspace{0.2cm}

\begin{center}
\Large{\HeadTitleTwo}
\end{center}

\hspace{3cm}

\begin{center}
Luc Ramsès TALLA WAFFO \\
Technische Universität Darmstadt\\
Karolinenplatz 5, 64289 Darmstadt, Germany\\
ramses.talla@stud.tu-darmstadt.de\\
\vspace{0.5cm}
\today
\end{center}

\begin{abstract}
Let \(g_n(x)\) be the Mittag--Leffler polynomials defined by

$$
\sum_{n\ge0}g_n(x)t^n=\frac12\left(\frac{1+t}{1-t}\right)^x.
$$

We derive the coefficient formula

$$
g_n(x)=\frac1n\sum_{j=0}^{\lfloor (n-1)/2\rfloor}
(-1)^j\frac{2^{n-2j-1}}{(n-2j-1)!}
[u^{2j}](u\cot u)^n\,x^{n-2j},
$$

equivalently

$$
[u^{n-r}](u\cot u)^n
=(-1)^{(n-r)/2}\frac{n(r-1)!}{2^{r-1}}[x^r]g_n(x).
$$

The identity is suggested by comparing an umbral representation of the hyperbolic tangent integral

$$
K(a,b)=\int_0^\infty \frac{\tanh^a x}{x^b}\,dx
$$

with a compact cotangent-coefficient formula, but is proved independently from the generating function by formal Lagrange--Bürmann inversion. Substitution of the proved bridge back into the integral formula then yields a rigorous derivation of the umbral Mittag--Leffler representation.

We also develop a Dirichlet--beta analogue for

$$
I_{m,n}^{\beta}=\int_0^\infty\frac{\tanh^m x}{x^n\cosh x}\,dx.
$$

Defining \(q_m(x)\) by

$$
\sum_{m\ge0}q_m(x)t^m=(1-t^2)^{-1/2}\exp(2x\operatorname{artanh}t),
$$

we prove

$$
[u^{m-r}]\frac{u}{\sin u}(u\cot u)^m
=(-1)^{(m-r)/2}\frac{r!}{2^r}[x^r]q_m(x).
$$

This gives a parallel umbral representation for the beta integrals. The family \(q_m\) is a shifted \(c=-1,\beta=1\) specialization of the classical Meixner family, yielding a beta/Meixner counterpart to the zeta/Mittag--Leffler correspondence.

\end{abstract}

\vspace{0.2cm}

\paragraph{Notation.}
We write $[z^k]F(z)$ for the coefficient of $z^k$ in the formal power
series $F(z)$ at the origin.  The Riemann zeta function and Dirichlet
beta function are denoted by $\zeta(s)$ and $\beta(s)$, respectively.
Throughout, $a,b,m,n,r,j,p$ are integers in the ranges specified where
they occur.  All applications of Lagrange inversion and formal residue
substitution are understood as identities of formal power series or
Laurent series; see \cite[Chap.~III]{Comtet1974},
\cite[Sec.~I.5]{FlajoletSedgewick2009}, and also
\cite{Gessel2016,SuryaWarnke2023}.

\vspace{0.5cm}

\section*{Introduction}

The Mittag--Leffler polynomials form a classical polynomial family going
back to Mittag--Leffler and Bateman; see \cite{MittagLeffler1891,Bateman1940}.
We use the normalization
\begin{equation}
\label{eq:generating-function}
\sum_{n=0}^{\infty}g_n(x)t^n
=\frac12\left(\frac{1+t}{1-t}\right)^x
=\frac12\exp\!\bigl(2x\operatorname{artanh}t\bigr),
\end{equation}
so that
\(g_0(x)=\frac12,\qquad g_1(x)=x,\qquad g_2(x)=x^2,\qquad g_3(x)=\dfrac{x+2x^3}{3}\).
Equivalently, if $M_n(x)$ denotes Bateman's factorial normalization,
then
\(M_n(x)=2n!\,g_n(x)\).

A particularly interesting appearance of these polynomials occurs in
improper integrals involving powers of $\tanh x$.  For \( \displaystyle K(a,b):=\int_0^\infty\frac{\tanh^a x}{x^b}\,dx\), with \(a\ge b\ge2\) and \(a\equiv b\pmod2\),
an umbral expression in terms of the Mittag--Leffler polynomial $g_a$
provides a useful point of comparison.  In our earlier work
\cite[Remark~4.7]{TallaWaffo2026}, we noted the existence of such umbral
Mittag--Leffler forms and evaluated the same integral family in direct
coefficient form, while Li and Chu treat it by contour integration and
obtain finite combinations of odd zeta values \cite{LiChu2024}.

The umbral formula is used here only as a point of departure.  The logical
core of the present
paper is independent of it: the main cotangent coefficient formula for
$g_n$ is proved directly from \eqref{eq:generating-function}.  At the end
of the paper the implication is reversed, and the umbral tangent-integral
formula is deduced rigorously from our compact coefficient formula
\cite{TallaWaffo2026}.  This produces a closed logical circle rather than
relying on a comparison of special values.

A second coefficient language occurs in the Dirichlet--beta counterpart
of the same hyperbolic framework, where the elementary kernel
\(\dfrac{u}{\sin u}(u\cot u)^m\)
replaces $(u\cot u)^n$.  Even beta values and their arithmetic have been
studied from several complementary viewpoints; see, for example,
\cite{Apostol1976,WeissteinBeta,Zudilin2019,Kyrion2025}, together with
our earlier studies \cite{talla_waffo_integral_2025,
TallaWaffo2026arxiv2602.16761}.  In \cref{sec:beta-umbral}
we show that this second kernel also possesses a natural polynomial
model.  The resulting family is a shifted formal specialization of the
Meixner polynomials, whose generating functions and hypergeometric
normalizations are classical; see \cite{NIST:DLMF,KoekoekLeskySwarttouw2010}
and the general Sheffer perspective in
\cite{Sheffer1939,RotaKahanerOdlyzko1973,Roman1984}.

The main identity is
\begin{equation}
\label{eq:main-intro}
 g_n(x)
 =\frac1n
 \sum_{j=0}^{\lfloor(n-1)/2\rfloor}
 (-1)^j
 \frac{2^{n-2j-1}}{(n-2j-1)!}
 [u^{2j}](u\cot u)^n\,x^{n-2j}.
\end{equation}
It gives every coefficient of $g_n$ from a single elementary generating
kernel $(u\cot u)^n$.  In particular, parity, degree, and the leading
coefficient are immediate.

In the sources surveyed for this paper, including the classical treatment of
Bateman and recent work on the same hyperbolic-integral family, we did not
locate this exact cotangent-coefficient representation for $g_n$.  We therefore
state the formula as a derived representation rather than making an absolute
priority claim.

The paper is organized as follows.  \Cref{sec:background} fixes the
normalizations and records the two tangent-integral representations that
motivate the formula.  \Cref{sec:formal-comparison} explains the formal
coefficient comparison and isolates the expected bridge between
$g_n(x)$ and $(u\cot u)^n$.  \Cref{sec:direct-proof} proves this bridge
independently from the generating function using Lagrange--B\"urmann
inversion.  \Cref{sec:consequences} gives structural consequences and
examples.  \Cref{sec:umbral-recovery} inserts the proved bridge into the
compact tangent-integral identity and obtains the umbral formula as a
theorem.  \Cref{sec:beta-umbral} then develops the Dirichlet--beta
companion: a shifted Meixner-type polynomial family is derived directly
from its generating function, its coefficient bridge is proved by a
formal residue substitution, and the compact beta formula is converted
into a second umbral identity.

\vspace{0.3cm}

\section{Background and the umbral starting point}
\label[section]{sec:background}

\subsection{Mittag--Leffler polynomials}

We begin with two elementary consequences of
\eqref{eq:generating-function}.

\begin{lemma}[Parity]
\label[lemma]{lem:parity}
For every $n\ge0$,
\(g_n(-x)=(-1)^n g_n(x)\).
Consequently, for $n\ge1$,
\[g_n(x)= \sum_{\substack{1\le r\le n\\r\equiv n\ (2)}} c_{n,r}x^r, \qquad c_{n,r}:=[x^r]g_n(x)\].
\end{lemma}

\begin{proof}
Replacing $(x,t)$ by $(-x,-t)$ in
\eqref{eq:generating-function} leaves the right-hand side unchanged:
\(\frac12\left(\frac{1-t}{1+t}\right)^{-x} =\frac12\left(\frac{1+t}{1-t}\right)^x\).
Comparison of coefficients of $t^n$ yields the parity identity.  Since
$g_n(0)=0$ for $n\ge1$, the asserted expansion follows.
\end{proof}

The classical relation with the factorial normalization is
\begin{equation}
\label{eq:M-g-relation}
M_n(x)=2n!\,g_n(x),
\end{equation}
consistent with the generating function
\(\sum_{n\ge0}M_n(x)\frac{t^n}{n!} =\left(\frac{1+t}{1-t}\right)^x\).
For historical background and integral representations of the
Mittag--Leffler polynomials, see \cite{Bateman1940}.

\subsection{The umbral hyperbolic tangent formula}

For odd positive integers $k$, define
\begin{equation}
\label{eq:hk-definition}
 h_k
 :=(-1)^{(k-1)/2}
 \frac{(k-1)!(2^k-1)}{2^{k-1}\pi^{k-1}}\zeta(k).
\end{equation}
It is convenient to make the umbral convention completely formal, in the
spirit of the linear-functional viewpoint of umbral calculus
\cite{Roman1984}.  Let $\mathcal U_h$ be the linear evaluation rule on
the monomials that occur below, defined by
\(\mathcal U_h(h^k)=h_k \qquad(k\ \text{odd})\).
Because $a\equiv b\pmod2$ and every nonzero exponent $r$ in $g_a$ has
$r\equiv a\pmod2$, the exponent $r+b-1$ is always odd.  Hence no even
umbra are needed.

The umbral representation that motivates our comparison is
\begin{equation}
\label{eq:umbral-start}
K(a,b)
=\frac{a2^{b-1}}{(b-1)!}
\mathcal U_h\!\left((-h)^{b-1}g_a(h)\right),
\qquad a\ge b\ge2,
\quad a\equiv b\pmod2.
\end{equation}
In our earlier work \cite[Remark~4.7]{TallaWaffo2026}, we noted the
existence of umbral Mittag--Leffler forms for this integral family.  Here
\eqref{eq:umbral-start} is used only to motivate the comparison; it is not
used as an input to the proof of our main polynomial identity.

Writing
\(g_a(x)= \sum_{\substack{1\le r\le a\\r\equiv a\ (2)}}c_{a,r}x^r\),
formula \eqref{eq:umbral-start} means explicitly
\begin{equation}
\label{eq:umbral-expanded}
K(a,b)
=\frac{a2^{b-1}}{(b-1)!}(-1)^{b-1}
\sum_{\substack{1\le r\le a\\r\equiv a\ (2)}}
 c_{a,r}\,h_{r+b-1}.
\end{equation}

\subsection{The compact cotangent formula}

In our earlier work \cite{TallaWaffo2026}, we proved that, for integers
$m\ge n\ge1$ with $m+n$ even,
\begin{equation}
\label{eq:talla-original}
\int_0^\infty\frac{\tanh^{m+1}x}{x^{n+1}}\,dx
=
(-1)^{(m-n)/2}
\sum_{p=\lceil n/2\rceil}^{(m+n)/2}
\binom{2p}{n}(2^{2p+1}-1)
\frac{\zeta(2p+1)}{\pi^{2p}}
[u^{m+n-2p}](u\cot u)^{m+1}.
\end{equation}
With
\(a=m+1,\qquad b=n+1\),
this becomes
\begin{equation}
\label{eq:talla-ab}
\boxed{
K(a,b)
=
(-1)^{(a-b)/2}
\sum_{p=\lceil(b-1)/2\rceil}^{(a+b-2)/2}
\binom{2p}{b-1}(2^{2p+1}-1)
\frac{\zeta(2p+1)}{\pi^{2p}}
[u^{a+b-2-2p}](u\cot u)^a.
}
\end{equation}
Unlike \eqref{eq:umbral-start}, this is a direct coefficient extraction
from an elementary power series.

\section{Formal comparison and the coefficient bridge}
\label[section]{sec:formal-comparison}

This section explains how the main identity is discovered.  The argument
is intentionally separated from the proof in
\cref{sec:direct-proof}.

Fix $a\ge b\ge2$ with $a\equiv b\pmod2$, and write
\(g_a(x)=\sum_r c_{a,r}x^r\).
For a summation index $p$ in \eqref{eq:talla-ab}, set
\begin{equation}
\label{eq:r-p-change}
r:=2p-b+2.
\end{equation}
Then
\(r+b-1=2p+1, \qquad a+b-2-2p=a-r\).
Thus the zeta value $\zeta(2p+1)$ in the compact formula corresponds to
the umbral value $h_{r+b-1}$ in \eqref{eq:umbral-expanded}.

Using \eqref{eq:hk-definition},
\(\displaystyle h_{2p+1} =(-1)^p \frac{(2p)!(2^{2p+1}-1)}{2^{2p}\pi^{2p}} \zeta(2p+1)\).
If one formally matches the multiplier of the same displayed zeta value
in \eqref{eq:umbral-expanded} and \eqref{eq:talla-ab}, the elementary
factors cancel and suggest
\begin{equation}
\label{eq:bridge-suggested}
[u^{a-r}](u\cot u)^a
=
(-1)^{(a-r)/2}
\frac{a(r-1)!}{2^{r-1}}c_{a,r}.
\end{equation}
Equivalently,
\begin{equation}
\label{eq:bridge-suggested-reverse}
c_{a,r}
=
(-1)^{(a-r)/2}
\frac{2^{r-1}}{a(r-1)!}
[u^{a-r}](u\cot u)^a.
\end{equation}

\begin{remark}[Why this is not yet a proof]
The equality of two numerical finite sums involving
$\zeta(3),\zeta(5),\ldots$ does not justify termwise coefficient
comparison: sufficiently strong linear independence statements for odd
zeta values are not known.  Therefore
\eqref{eq:bridge-suggested} is used only as a discovery mechanism.  The
next section proves it independently as an identity of formal power
series, with no zeta values and no integral evaluations involved.
\end{remark}

Substituting $r=a-2j$ in \eqref{eq:bridge-suggested-reverse} suggests the
polynomial identity
\begin{equation}
\label{eq:main-suggested}
 g_a(x)
 =\frac1a
 \sum_{j=0}^{\lfloor(a-1)/2\rfloor}
 (-1)^j
 \frac{2^{a-2j-1}}{(a-2j-1)!}
 [u^{2j}](u\cot u)^a\,x^{a-2j}.
\end{equation}
We now prove this directly.

\section{Independent derivation from the generating function}
\label[section]{sec:direct-proof}

The proof is a short application of Lagrange--B\"urmann inversion to the
inverse pair
\(w=\operatorname{artanh}t, \qquad t=\tanh w\).

\begin{lemma}[Coefficient of a fixed power of $x$]
\label[lemma]{lem:x-coefficient-artanh}
Let $n\ge1$ and $1\le r\le n$.  Then
\begin{equation}
\label{eq:x-coefficient-artanh}
[x^r]g_n(x)
=\frac{2^{r-1}}{r!}
[t^n](\operatorname{artanh}t)^r.
\end{equation}
\end{lemma}

\begin{proof}
Expanding the exponential in \eqref{eq:generating-function} gives
\[
\frac12\exp\!\bigl(2x\operatorname{artanh}t\bigr)
=
\frac12\sum_{r\ge0}
\frac{2^r x^r}{r!}(\operatorname{artanh}t)^r.
\]
Extracting $[x^rt^n]$ gives \eqref{eq:x-coefficient-artanh}.
\end{proof}

\begin{lemma}[Lagrange--B\"urmann conversion]
\label[lemma]{lem:lagrange-conversion}
For $n\ge r\ge1$,
\begin{equation}
\label{eq:lagrange-conversion}
[t^n](\operatorname{artanh}t)^r
=\frac{r}{n}
[w^{n-r}](w\coth w)^n.
\end{equation}
\end{lemma}

\begin{proof}
The series $t=\tanh w$ has linear term $w$, hence possesses the formal
inverse $w=\operatorname{artanh}t$.  The Lagrange--B\"urmann formula in
the standard power form gives
\[
[t^n]w(t)^r
=\frac{r}{n}[w^{n-r}]
\left(\frac{w}{\tanh w}\right)^n.
\]
Since $w/\tanh w=w\coth w$, this is
\eqref{eq:lagrange-conversion}.
\end{proof}

Combining the two lemmas yields an intermediate hyperbolic coefficient
formula.

\begin{proposition}[Hyperbolic coefficient bridge]
\label[proposition]{prop:coth-bridge}
For $n\ge1$ and $1\le r\le n$,
\begin{equation}
\label{eq:coth-bridge}
[x^r]g_n(x)
=
\frac{2^{r-1}}{n(r-1)!}
[w^{n-r}](w\coth w)^n.
\end{equation}
In particular the coefficient vanishes whenever $n-r$ is odd.
\end{proposition}

\begin{proof}
Insert \eqref{eq:lagrange-conversion} into
\eqref{eq:x-coefficient-artanh}.  Since $w\coth w$ is an even formal
power series, $[w^{n-r}](w\coth w)^n=0$ if $n-r$ is odd.
\end{proof}

To pass from the hyperbolic to the trigonometric kernel, observe that
\begin{equation}
\label{eq:coth-cot-analytic-continuation}
(iu)\coth(iu)=u\cot u.
\end{equation}
Thus, if $n-r=2j$,
\begin{equation}
\label{eq:coth-cot-coefficients}
[w^{2j}](w\coth w)^n
=(-1)^j[u^{2j}](u\cot u)^n.
\end{equation}

\begin{theorem}[Cotangent coefficient formula for $g_n$]
\label[theorem]{thm:main-mittag-leffler}
For every integer $n\ge1$,
\begin{equation}
\label{eq:main-theorem}
\boxed{
 g_n(x)
 =\frac1n
 \sum_{j=0}^{\lfloor(n-1)/2\rfloor}
 (-1)^j
 \frac{2^{\,n-2j-1}}{(n-2j-1)!}
 [u^{2j}](u\cot u)^n\,x^{n-2j}.
}
\end{equation}
Equivalently, for $1\le r\le n$ with $r\equiv n\pmod2$,
\begin{equation}
\label{eq:main-bridge}
\boxed{
 [u^{n-r}](u\cot u)^n
 =(-1)^{(n-r)/2}
 \frac{n(r-1)!}{2^{r-1}}
 [x^r]g_n(x).
}
\end{equation}
\end{theorem}

\begin{proof}
Set $r=n-2j$ in \eqref{eq:coth-bridge} and use
\eqref{eq:coth-cot-coefficients}.  This gives
\[
[x^{n-2j}]g_n(x)
=
\frac{(-1)^j}{n}
\frac{2^{n-2j-1}}{(n-2j-1)!}
[u^{2j}](u\cot u)^n.
\]
Summing over all admissible $j$ proves \eqref{eq:main-theorem}, and
rearranging the same identity gives \eqref{eq:main-bridge}.
\end{proof}

\begin{corollary}[Bateman normalization]
\label[corollary]{cor:bateman-normalization}
For $n\ge1$, the factorially normalized Mittag--Leffler polynomial
$M_n(x)=2n!g_n(x)$ satisfies
\begin{equation}
\label{eq:bateman-cotangent}
\boxed{
M_n(x)
=(n-1)!
\sum_{j=0}^{\lfloor(n-1)/2\rfloor}
(-1)^j
\frac{2^{n-2j}}{(n-2j-1)!}
[u^{2j}](u\cot u)^n\,x^{n-2j}.
}
\end{equation}
\end{corollary}

\begin{proof}
Multiply \eqref{eq:main-theorem} by $2n!$ and use
$n!/n=(n-1)!$.
\end{proof}

\section{Structural consequences and examples}
\label[section]{sec:consequences}

\subsection{Degree, parity, and endpoint coefficients}

The main formula makes several familiar properties of the
Mittag--Leffler polynomials immediate.

\begin{corollary}[Leading coefficient]
\label[corollary]{cor:leading-coefficient}
For every $n\ge1$,
\([x^n]g_n(x)=\frac{2^{n-1}}{n!}\).
\end{corollary}

\begin{proof}
In \eqref{eq:main-theorem}, the coefficient of $x^n$ corresponds to
$j=0$.  Since $[u^0](u\cot u)^n=1$,
\([x^n]g_n(x)=\frac1n\frac{2^{n-1}}{(n-1)!} =\frac{2^{n-1}}{n!}\).
\end{proof}

\begin{corollary}[The linear coefficient in odd degree]
\label[corollary]{cor:linear-coefficient}
If $n$ is odd, then
\([x]g_n(x)=\frac1n\),
and consequently
\begin{equation}
\label{eq:terminal-cotangent-coefficient}
[u^{n-1}](u\cot u)^n=(-1)^{(n-1)/2}.
\end{equation}
\end{corollary}

\begin{proof}
From \eqref{eq:x-coefficient-artanh} with $r=1$,
\([x]g_n(x)=[t^n]\operatorname{artanh}t\).
For odd $n$, the right-hand side is $1/n$.  Formula
\eqref{eq:terminal-cotangent-coefficient} then follows from
\eqref{eq:main-bridge} with $r=1$.
\end{proof}

The parity statement in \cref{lem:parity} is also transparent from
\eqref{eq:main-theorem}: only the powers
\(x^n,x^{n-2},x^{n-4},\ldots\)
appear.

\subsection{Example: $g_5$}

The needed expansion is
\((u\cot u)^5 =1-\frac53u^2+u^4+O(u^6)\).
Formula \eqref{eq:main-theorem} gives
\begin{align*}
g_5(x)
&=\frac15\left(
\frac{2^4}{4!}x^5
-\frac{2^2}{2!}\left(-\frac53\right)x^3
+[u^4](u\cot u)^5x
\right)\\
&=\frac15\left(\frac23x^5+\frac{10}{3}x^3+x\right)\\
&=\boxed{\frac{3x+10x^3+2x^5}{15}}.
\end{align*}

\subsection{Example: $g_6$}

Here
\((u\cot u)^6 =1-2u^2+\frac{23}{15}u^4+O(u^6)\).
Hence
\begin{align*}
g_6(x)
&=\frac16\left(
\frac{2^5}{5!}x^6
-\frac{2^3}{3!}(-2)x^4
+\frac{2}{1!}\frac{23}{15}x^2
\right)\\
&=\boxed{\frac{23x^2+20x^4+2x^6}{45}}.
\end{align*}
These agree with the classical low-degree polynomials recorded, for
example, in \cite{Bateman1940}.

\subsection{A normalization identity for powers of $u\cot u$}

For $n\ge1$, one has $g_n(1)=1$.  Setting $x=1$ in
\eqref{eq:main-theorem} therefore gives the finite identity
\begin{equation}
\label{eq:normalization-cotangent-sum}
\boxed{
\sum_{j=0}^{\lfloor(n-1)/2\rfloor}
(-1)^j
\frac{2^{n-2j-1}}{(n-2j-1)!}
[u^{2j}](u\cot u)^n
=n.
}
\end{equation}
Thus the usual normalization of the Mittag--Leffler polynomials becomes
a nontrivial coefficient sum for the cotangent kernel.

\section{Recovery of the umbral tangent-integral formula}
\label[section]{sec:umbral-recovery}

We now close the circle.  The cotangent coefficient bridge was proved in
\cref{sec:direct-proof} without using the umbral tangent-integral
formula.  We may therefore insert it into our compact coefficient formula
\eqref{eq:talla-ab} and deduce the umbral representation rigorously.

\begin{theorem}[Umbral Mittag--Leffler form of $K(a,b)$]
\label[theorem]{thm:umbral-recovered}
Let $a\ge b\ge2$ be integers with $a\equiv b\pmod2$, and define $h_k$
for odd $k$ by \eqref{eq:hk-definition}.  Then
\begin{equation}
\label{eq:umbral-recovered}
\boxed{
K(a,b)
=\int_0^\infty\frac{\tanh^a x}{x^b}\,dx
=\frac{a2^{b-1}}{(b-1)!}
\mathcal U_h\!\left((-h)^{b-1}g_a(h)\right).
}
\end{equation}
\end{theorem}

\begin{proof}
Start from \eqref{eq:talla-ab}.  For each summation index $p$, put
\(r=2p-b+2\).
Then
\(2p+1=r+b-1, \qquad a+b-2-2p=a-r, \qquad r-1=2p-b+1\).
By \eqref{eq:main-bridge},
\begin{equation}
\label{eq:insert-bridge}
[u^{a-r}](u\cot u)^a
=(-1)^{(a-r)/2}
\frac{a(r-1)!}{2^{r-1}}c_{a,r},
\qquad
c_{a,r}=[x^r]g_a(x).
\end{equation}
Also,
\begin{equation}
\label{eq:binomial-collapse}
\binom{2p}{b-1}(r-1)!
=\frac{(2p)!}{(b-1)!},
\end{equation}
and, because $r-1=2p-b+1$,
\begin{equation}
\label{eq:power-collapse}
\frac1{2^{r-1}}=\frac{2^{b-1}}{2^{2p}}.
\end{equation}
Finally, the signs satisfy
\begin{equation}
\label{eq:sign-collapse}
(-1)^{(a-b)/2+(a-r)/2}
=(-1)^{b-1+p},
\end{equation}
because the difference of the two exponents is
$a-r-2b+2$, an even integer.

Substituting \eqref{eq:insert-bridge}--\eqref{eq:sign-collapse} into the
$p$th term of \eqref{eq:talla-ab} gives
\begin{align*}
&\frac{a2^{b-1}}{(b-1)!}
(-1)^{b-1}c_{a,r}
\left(
\frac{(-1)^p(2p)!(2^{2p+1}-1)}{2^{2p}\pi^{2p}}
\zeta(2p+1)
\right)\\
&\qquad
=\frac{a2^{b-1}}{(b-1)!}
(-1)^{b-1}c_{a,r}h_{2p+1}\\
&\qquad
=\frac{a2^{b-1}}{(b-1)!}
(-1)^{b-1}c_{a,r}h_{r+b-1}.
\end{align*}
Summing over $p$, equivalently over all nonzero powers $r$ of $g_a$,
we obtain
\[
K(a,b)
=\frac{a2^{b-1}}{(b-1)!}(-1)^{b-1}
\sum_r c_{a,r}h_{r+b-1},
\]
which is exactly the umbral evaluation
\[
\frac{a2^{b-1}}{(b-1)!}
\mathcal U_h\!\left((-h)^{b-1}g_a(h)\right).
\]
\end{proof}

\begin{remark}[Logical direction]
The proof of \cref{thm:umbral-recovered} does not compare coefficients of
odd zeta values.  Instead, the coefficient bridge
\eqref{eq:main-bridge} has already been established formally from the
generating function, and is substituted directly into the proven
integral identity \eqref{eq:talla-ab}.  No unproved linear independence
property of odd zeta values is required.
\end{remark}

\subsection{Consistency check: $K(5,3)$}

From
\(g_5(x)=\dfrac{3x+10x^3+2x^5}{15}\)
and \eqref{eq:umbral-recovered},
\(K(5,3) =\frac23\left(3h_3+10h_5+2h_7\right)\).
Using
\[
h_3=-\frac72\frac{\zeta(3)}{\pi^2},
\qquad
h_5=\frac{93}{2}\frac{\zeta(5)}{\pi^4},
\qquad
h_7=-\frac{5715}{4}\frac{\zeta(7)}{\pi^6},
\]
we recover
\begin{equation}
\label{eq:K53}
\boxed{
K(5,3)
=-7\frac{\zeta(3)}{\pi^2}
+310\frac{\zeta(5)}{\pi^4}
-1905\frac{\zeta(7)}{\pi^6}.
}
\end{equation}
This agrees with the explicit evaluations of Li and Chu
\cite{LiChu2024} and follows directly from our compact coefficient formula
\cite{TallaWaffo2026}.  In particular, the calculation fixes the sign of
the intermediate umbral combination without appealing to a secondary
example table.

\section{A Dirichlet--beta umbral companion}
\label[section]{sec:beta-umbral}

We now apply the same principle to the second hyperbolic integral family
\begin{equation}
\label{eq:beta-integral-family}
I_{m,n}^{\beta}
:=\int_0^\infty\frac{\tanh^m x}{x^n\cosh x}\,dx,
\qquad
m\ge n\ge1,
\qquad
m+n\equiv0\pmod2.
\end{equation}
The exponential decay of $1/\cosh x$ controls the integral at infinity,
while $\tanh^m x/x^n=O(x^{m-n})$ at the origin.  The Dirichlet beta
function itself and the arithmetic of its even values are discussed in
standard and recent sources such as
\cite{Apostol1976,WeissteinBeta,Zudilin2019,Kyrion2025}, as well as in
our earlier studies \cite{talla_waffo_integral_2025,
TallaWaffo2026arxiv2602.16761}.

The compact coefficient identity used below is the beta-type counterpart
of \eqref{eq:talla-original}; we proved it in our recent work
\cite{TallaWaffo2026BetaAnalogues}:
\begin{equation}
\label{eq:beta-compact}
\boxed{
I_{m,n}^{\beta}
=
(-1)^{(m-n)/2}
\sum_{p=\lceil n/2\rceil}^{(m+n)/2}
2^{2p}\binom{2p-1}{n-1}
\frac{\beta(2p)}{\pi^{2p-1}}
[u^{m+n-2p}]
\frac{u}{\sin u}(u\cot u)^m.
}
\end{equation}
We use \eqref{eq:beta-compact} only after proving the polynomial bridge
below independently.  This separation mirrors the logical strategy used
for the zeta-type integral family and avoids deriving a polynomial identity
from a termwise comparison of special values.

\subsection{The shifted Meixner-type polynomial family}

Define polynomials $q_m(x)$ by
\begin{equation}
\label{eq:q-generating-function}
\boxed{
\sum_{m=0}^{\infty}q_m(x)t^m
=
\frac{1}{\sqrt{1-t^2}}
\exp\!\bigl(2x\operatorname{artanh}t\bigr)
=
(1+t)^{x-1/2}(1-t)^{-x-1/2}.
}
\end{equation}
The standard Meixner generating function can be written as
\[
\left(1-\frac{t}{c}\right)^y(1-t)^{-y-\gamma}
=
\sum_{m\ge0}\frac{(\gamma)_m}{m!}
M_m(y;\gamma,c)t^m;
\]
see Meixner's original paper \cite{Meixner1934}, the modern treatments
\cite{Ismail2009,KoekoekLeskySwarttouw2010}, and
\cite[Sec.~18.23]{NIST:DLMF}.  For historical context on Meixner's
classification and its later reception, see also \cite{ButzerKoornwinder2019}.
Thus
\eqref{eq:q-generating-function} is obtained formally by taking
\(y=x-\frac12,\qquad \gamma=1,\qquad c=-1\).
Since $c=-1$ lies outside the usual positive orthogonality regime, we
use the Meixner identification only as a polynomial/generating-function
specialization, not as a claim of positive orthogonality.  This formal
specialization is also compatible with the general Sheffer viewpoint;
see \cite{Sheffer1939,RotaKahanerOdlyzko1973,Roman1984,NIST:DLMF}.

The first few members are \(q_0(x)=1\), \(q_1(x)=2x\), \(q_2(x)=2x^2+\frac12\), \(q_3(x)=\frac43x^3+\frac53x\), and \(q_4(x)=\frac23x^4+\frac73x^2+\frac38\).
Moreover,
\(q_m(-x)=(-1)^m q_m(x)\),
because replacing $(x,t)$ by $(-x,-t)$ in
\eqref{eq:q-generating-function} leaves the right-hand side unchanged.
Hence we may write
\begin{equation}
\label{eq:q-expansion}
q_m(x)=
\sum_{\substack{0\le r\le m\\r\equiv m\ (2)}}d_{m,r}x^r,
\qquad
d_{m,r}:=[x^r]q_m(x).
\end{equation}

\begin{lemma}[Fixed-power coefficient for $q_m$]
\label[lemma]{lem:q-fixed-power}
For $m\ge0$ and $0\le r\le m$,
\begin{equation}
\label{eq:q-fixed-power}
[x^r]q_m(x)
=
\frac{2^r}{r!}
[t^m]\frac{(\operatorname{artanh}t)^r}{\sqrt{1-t^2}}.
\end{equation}
\end{lemma}

\begin{proof}
Expand the exponential in \eqref{eq:q-generating-function} and extract
$[x^rt^m]$.  This gives \eqref{eq:q-fixed-power} directly.
\end{proof}

The additional factor $(1-t^2)^{-1/2}$ changes the Lagrange conversion
slightly.  A formal residue substitution is particularly efficient
here.  Such residue forms of Lagrange inversion are standard; see
\cite{Comtet1974,FlajoletSedgewick2009,Gessel2016,SuryaWarnke2023}.

\begin{lemma}[Hyperbolic beta-kernel conversion]
\label[lemma]{lem:q-hyperbolic-conversion}
For $m\ge r\ge0$,
\begin{equation}
\label{eq:q-hyperbolic-conversion}
[t^m]\frac{(\operatorname{artanh}t)^r}{\sqrt{1-t^2}}
=
[w^{m-r}]\operatorname{sech}w\,(w\coth w)^{m+1}.
\end{equation}
\end{lemma}

\begin{proof}
Put $t=\tanh w$, so that $w=\operatorname{artanh}t$ and
$dt=\operatorname{sech}^2w\,dw$.  Formal residue substitution gives
\begin{align*}
[t^m]\frac{(\operatorname{artanh}t)^r}{\sqrt{1-t^2}}
&=\operatorname*{Res}_{t=0}
\frac{(\operatorname{artanh}t)^r}{\sqrt{1-t^2}}
\frac{dt}{t^{m+1}}\\
&=\operatorname*{Res}_{w=0}
\frac{w^r\operatorname{sech}w}{\tanh^{m+1}w}\,dw\\
&=[w^{m-r}]\operatorname{sech}w\,(w\coth w)^{m+1}.
\end{align*}
No analytic convergence assertion is needed: the calculation takes
place in formal Laurent series around the origin.
\end{proof}

Now use
\(\operatorname{sech}(iu)=\sec u, \qquad (iu)\coth(iu)=u\cot u\),
and the elementary identity
\begin{equation}
\label{eq:beta-kernel-factorization}
\sec u\,(u\cot u)^{m+1}
=
\frac{u}{\sin u}(u\cot u)^m.
\end{equation}
Both sides of \eqref{eq:beta-kernel-factorization} are even formal
series.

\begin{theorem}[Dirichlet--beta coefficient bridge]
\label[theorem]{thm:beta-coefficient-bridge}
For $m\ge0$ and $0\le r\le m$ with $r\equiv m\pmod2$,
\begin{equation}
\label{eq:beta-coefficient-bridge}
\boxed{
[u^{m-r}]\frac{u}{\sin u}(u\cot u)^m
=
(-1)^{(m-r)/2}\frac{r!}{2^r}[x^r]q_m(x).
}
\end{equation}
Equivalently,
\begin{equation}
\label{eq:q-beta-coefficient-formula}
\boxed{
q_m(x)
=
\sum_{j=0}^{\lfloor m/2\rfloor}
(-1)^j
\frac{2^{m-2j}}{(m-2j)!}
[u^{2j}]\frac{u}{\sin u}(u\cot u)^m
x^{m-2j}.
}
\end{equation}
\end{theorem}

\begin{proof}
Combining \eqref{eq:q-fixed-power} and
\eqref{eq:q-hyperbolic-conversion} yields
\[
[x^r]q_m(x)
=\frac{2^r}{r!}
[w^{m-r}]\operatorname{sech}w\,(w\coth w)^{m+1}.
\]
If $m-r=2j$, substituting $w=iu$ gives
\[
[w^{m-r}]\operatorname{sech}w\,(w\coth w)^{m+1}
=(-1)^j
[u^{m-r}]\sec u\,(u\cot u)^{m+1}.
\]
Use \eqref{eq:beta-kernel-factorization} and rearrange to obtain
\eqref{eq:beta-coefficient-bridge}.  Setting $r=m-2j$ and summing over
all admissible powers proves \eqref{eq:q-beta-coefficient-formula}.
\end{proof}

Thus the beta kernel has exactly the same status for $q_m$ that the
cotangent kernel has for the Mittag--Leffler polynomials in
\cref{thm:main-mittag-leffler}.  The first coefficient dictionary identifies
\([u^{n-r}](u\cot u)^n\) with \([x^r]g_n(x)\), while the second identifies
\([u^{m-r}]\frac{u}{\sin u}(u\cot u)^m\) with \([x^r]q_m(x)\).

\subsection{Recovery of the Dirichlet--beta umbral form}

For positive even integers $k$, define
\begin{equation}
\label{eq:beta-umbra-values}
b_k
:=
(-1)^{k/2}(k-1)!\frac{\beta(k)}{\pi^{k-1}}.
\end{equation}
As in \cref{sec:background}, we make the umbral notation formal through
a linear evaluation map in the sense of \cite{Roman1984}.  Let
$\mathcal U_\beta$ be defined on the even monomials that occur below by
\begin{equation}
\label{eq:beta-umbra-evaluation}
\mathcal U_\beta(B^k)=b_k
\qquad(k\ge2,\ k\ \text{even}).
\end{equation}
The parity assumptions guarantee that no odd beta umbra are needed.

\begin{theorem}[Umbral shifted-Meixner form of $I_{m,n}^{\beta}$]
\label[theorem]{thm:beta-umbral}
Let $m\ge n\ge1$ and $m+n\equiv0\pmod2$.  Then
\begin{equation}
\label{eq:beta-umbral}
\boxed{
I_{m,n}^{\beta}
=
\int_0^\infty\frac{\tanh^m x}{x^n\cosh x}\,dx
=
\frac{2^n}{(n-1)!}
\mathcal U_\beta\!\left((-B)^nq_m(B)\right).
}
\end{equation}
\end{theorem}

\begin{proof}
Write $q_m(x)=\sum_r d_{m,r}x^r$ as in
\eqref{eq:q-expansion}.  In the $p$th term of
\eqref{eq:beta-compact}, put
\(r:=2p-n\).
Then
\(2p=n+r, \qquad m+n-2p=m-r\).
By \eqref{eq:beta-coefficient-bridge},
\[
[u^{m-r}]\frac{u}{\sin u}(u\cot u)^m
=
(-1)^{(m-r)/2}\frac{r!}{2^r}d_{m,r}.
\]
Moreover,
\(\frac{2^{2p}}{2^r}=2^n, \qquad \binom{2p-1}{n-1}r! =\frac{(2p-1)!}{(n-1)!}\).
The total sign is
\((-1)^{(m-n)/2+(m-r)/2} =(-1)^{m-p} =(-1)^n(-1)^p\),
where the last equality uses $m\equiv n\pmod2$.  Consequently the
$p$th summand becomes
\[
\frac{2^n}{(n-1)!}
(-1)^n d_{m,r}
\left(
(-1)^p(2p-1)!\frac{\beta(2p)}{\pi^{2p-1}}
\right).
\]
By \eqref{eq:beta-umbra-values}, the parenthesis is $b_{2p}=b_{n+r}$.
Summing over $p$, equivalently over the admissible powers $r$ of
$q_m$, gives
\[
I_{m,n}^{\beta}
=\frac{2^n}{(n-1)!}(-1)^n
\sum_r d_{m,r}b_{n+r},
\]
which is precisely \eqref{eq:beta-umbral}.
\end{proof}

\begin{remark}[Logical direction in the beta case]
As in \cref{thm:umbral-recovered}, no linear independence statement for
special values is used.  The bridge
\eqref{eq:beta-coefficient-bridge} is derived solely from the generating
function \eqref{eq:q-generating-function}.  Only after that independent
proof is complete is the bridge inserted into the compact integral
identity \eqref{eq:beta-compact}.  Thus the beta umbral formula is a
formal consequence of two separately established identities rather
than a termwise comparison of the numbers
$\beta(2),\beta(4),\ldots$.
\end{remark}

\subsection{Consistency checks}

For $m=n=1$, we have $q_1(x)=2x$, so \(\frac{2^1}{0!}\mathcal U_\beta\!\left((-B)q_1(B)\right)=-4\mathcal U_\beta(B^2)=4\frac{\beta(2)}{\pi}\).
Hence
\begin{equation}
\label{eq:beta-check-11}
\boxed{
\int_0^\infty\frac{\tanh x}{x\cosh x}\,dx
=4\frac{\beta(2)}{\pi}.
}
\end{equation}
The value $\beta(2)$ is Catalan's constant; see
\cite{WeissteinBeta,Zudilin2019}.

For $m=n=2$, the polynomial $q_2(x)=2x^2+\tfrac12$ gives
\begin{align*}
I_{2,2}^{\beta}
&=4\mathcal U_\beta\!\left(B^2\left(2B^2+\frac12\right)\right)\\
&=8b_4+2b_2\\
&=48\frac{\beta(4)}{\pi^3}
-2\frac{\beta(2)}{\pi}.
\end{align*}
Therefore
\begin{equation}
\label{eq:beta-check-22}
\boxed{
\int_0^\infty\frac{\tanh^2x}{x^2\cosh x}\,dx
=48\frac{\beta(4)}{\pi^3}
-2\frac{\beta(2)}{\pi}.
}
\end{equation}
These two checks also follow immediately from
\eqref{eq:beta-compact} and verify the normalization and signs of
\eqref{eq:beta-umbral}.

\section{Concluding remarks}
\label[section]{sec:conclusion}

The identity
\[
 g_n(x)
 =\frac1n
 \sum_{j=0}^{\lfloor(n-1)/2\rfloor}
 (-1)^j
 \frac{2^{n-2j-1}}{(n-2j-1)!}
 [u^{2j}](u\cot u)^n x^{n-2j}
\]
provides a direct dictionary between two coefficient languages that
arise naturally in hyperbolic tangent integrals.  On one side are the
ordinary coefficients of the classical Mittag--Leffler polynomial
$g_n$; on the other are the even coefficients of the elementary kernel
$(u\cot u)^n$.

The comparison of the two integral formulae suggests the dictionary,
but the actual proof is intrinsic to the Mittag--Leffler generating
function.  Lagrange--B\"urmann inversion converts powers of
$\operatorname{artanh}t$ into coefficients of $(w\coth w)^n$, and the
substitution $w=iu$ converts these into coefficients of $(u\cot u)^n$.
The same bridge then transforms our direct tangent-integral formula
back into the umbral representation.  In this sense the
cotangent extraction is not merely another way to compute the same
integrals: it is an explicit coefficient model for the Mittag--Leffler
polynomials themselves.

The Dirichlet--beta companion of \cref{sec:beta-umbral} shows that the
coefficient-dictionary mechanism is not confined to the
Mittag--Leffler family.  The additional factor $u/\sin u$ is reflected
on the polynomial side by the additional factor
$(1-t^2)^{-1/2}$ in \eqref{eq:q-generating-function}; after formal
inversion this becomes the factor $\operatorname{sech}w$ in
\eqref{eq:q-hyperbolic-conversion}.  The resulting family is a shifted
$c=-1$, $\beta=1$ specialization of the Meixner generating function,
and \cref{thm:beta-umbral} places the even Dirichlet beta values into an
umbral formula parallel to the odd-zeta Mittag--Leffler formula.

Several natural extensions remain.  Since the coefficients of
$(u\cot u)^n$ and of $\dfrac{u}{\sin u}(u\cot u)^m$ admit descriptions
through Bernoulli- and Euler-type numbers, the two polynomial bridges
can be converted into finite special-number formulae.  More generally,
it is natural to ask which other Sheffer or classical hypergeometric
families arise when elementary factors are inserted into the inverse
hyperbolic generating kernel; compare the general frameworks in
\cite{Sheffer1939,RotaKahanerOdlyzko1973,Roman1984,Meixner1934,
Ismail2009,NIST:DLMF,KoekoekLeskySwarttouw2010}.

\section*{Acknowledgments}
The author acknowledges the use of an AI language model for language and
presentation assistance and for exploratory algebraic checks.  

\printbibliography

\end{document}